\documentclass{llncs}
\usepackage{graphicx}
\usepackage{amsmath}
\usepackage{amssymb}
\usepackage{enumerate}
\usepackage{varioref}
\usepackage{charter,eulervm}

\newcommand{\es}{\varnothing}

\title{\sc {On the strong chromatic index and maximum
induced matching
of tree-cographs, permutation graphs and chordal
bipartite graphs}}

\author{
 Ton~Kloks\inst{1}
\and
 Chin-Ting~Ung$^{\star,}$\inst{1}
\and
 Yue-Li~Wang\thanks{All correspondence should be
addressed to Professor Yue-Li Wang, Department of Information
Management, National Taiwan University of Science and Technology,
43, Section 4, Kee-Lung Road, Taipei, Taiwan 10607 (e-mail: ylwang@cs.ntust.edu.tw). }$^,$\inst{2}
}
\institute{
 Department of Computer Science\\
 National Tsing Hua University,
 No.~101, Sec.~2, Kuang Fu Rd., Hsinchu, Taiwan\\
 {\tt spoon@cs.nthu.edu.tw, wonderboy0915@gmail.com}
\and
 Department of Information Management\\
 National Taiwan University of Science and Technology\\
 No.~43, Sec.~4, Keelung Rd., Taipei, 106, Taiwan\\
 {\tt ylwang@cs.ntust.edu.tw}
}

\begin{document}

\maketitle

\begin{abstract}
We show that there exist linear-time algorithms that compute the
strong chromatic index and a maximum induced matching of
tree-cographs when the decomposition tree is a part of the input. We
also show that there exist efficient algorithms for the strong
chromatic index of (bipartite) permutation graphs and of chordal
bipartite graphs.
\end{abstract}

\setcounter{page}{0}
\newpage
\section{Introduction}

\begin{definition}[\cite{kn:cameron}]
An {\em induced matching\/} in a graph $G$ is a set of edges, no two
of which meet a common vertex or are joined by an edge of $G$. The
size of an induced matching is the number of edges in the induced
matching. An induced matching is maximum if its size is largest
among all possible induced matchings.
\end{definition}

\begin{definition}[\cite{kn:fouquet}]
Let $G=(V,E)$ be a graph. A {\em strong edge coloring\/} of $G$ is a
proper edge coloring such that no edge is adjacent to two edges of
the same color. A strong edge-coloring of a graph is a partition of
its edges into induced matchings. The {\em strong chromatic index\/}
of $G$ is the minimal integer $k$ such that $G$ has a strong edge
coloring with $k$ colors. We denote the strong chromatic index of
$G$ by $s\chi^{\prime}(G)$.
\end{definition}

Equivalently, a strong edge coloring of $G$ is a vertex coloring of
$L(G)^2$, the square of the linegraph of $G$. The strong chromatic
index problem can be solved in polynomial time for chordal
graphs~\cite{kn:cameron} and for partial
k-trees~\cite{kn:salavatipour}, and can be solved in linear time for
trees~\cite{kn:faudree}. However, it is NP-complete to find the
strong chromatic index for general graphs
\cite{kn:cameron,kn:mahdian,kn:stockmeyer} or even for planar
bipartite graphs~\cite{kn:hocquard}. In this paper, we show that
there exist linear-time algorithms that compute the strong chromatic
index and a maximum induced matching of tree-cographs when the
decomposition tree is a part of the input. We also show that there
exist efficient algorithms for the strong chromatic index of
(bipartite) permutation graphs and of chordal bipartite graphs.

The class of tree-cographs was introduced by
Tinhofer in~\cite{kn:tinhofer}.

\begin{definition}
Tree-cographs are defined recursively by the following rules.
\begin{enumerate}[\rm 1.]
\item Every tree is a tree-cograph.
\item If $G$ is a tree-cograph then also the
complement $\Bar{G}$ of $G$ is a tree-cograph.
\item For $k \geq 2$, if $G_1,\ldots,G_k$ are connected tree-cographs
then also the disjoint union is a tree-cograph.
\end{enumerate}
\end{definition}

Let $G$ be a tree-cograph. A decomposition tree for $G$ consists
of a rooted binary tree $T$ in which each
internal node, including the root,
is labeled as a join node $\otimes$ or a union node $\oplus$.
The leaves of $T$ are labeled by trees or complements of trees.
It is easy to see that a decomposition tree for a tree-cograph $G$
can be obtained in $O(n^3)$ time.

\section{The strong chromatic index of tree-cographs}

The {\em linegraph} $L(G)$ of a graph $G$ is the intersection graph
of the edges of $G$~\cite{kn:beineke}. It is well-known that, when
$G$ is a tree then the linegraph $L(G)$ of $G$ is a claw-free
blockgraph~\cite{kn:harary}. A graph is {\em chordal} if it has no
induced cycles of length more than three~\cite{kn:dirac}. Notice
that blockgraphs are chordal.

\medskip

A vertex $x$ in a graph $G$ is {\em simplicial} if its neighborhood
$N(x)$ induces a clique in $G$. Chordal graphs are characterized by
the property of having a {\em perfect elimination ordering}, which
is an ordering $[v_1,\ldots,v_n]$ of the vertices of $G$ such that
$v_i$ is simplicial in the graph induced by $\{v_i,\ldots,v_n\}$. A
perfect elimination ordering of a chordal graph can be computed in
linear time~\cite{kn:rose}. This implies that chordal graphs have at
most $n$ maximal cliques, and the clique number can be computed in
linear time, where the {\em clique number} of $G$, denoted by
$\omega(G)$, is the number of vertices in a maximum clique of $G$.

\begin{theorem}[\cite{kn:cameron}]
If $G$ is a chordal graph then $L(G)^2$ is also chordal.
\end{theorem}

\begin{theorem}[\cite{kn:cameron3}]
\label{weakly chordal}
Let $k \in \mathbb{N}$ and let $k \geq 4$.
Let $G$ be a graph and assume that $G$ has no induced cycles
of length at least $k$. Then $L(G)^2$ has no
induced cycles of length at least $k$.
\end{theorem}

\begin{lemma}
Tree-cographs have no induced cycles of length more than four.
\end{lemma}
\begin{proof} Let $G$ be a tree-cograph. First observe that trees are
bipartite. It follows that complements of trees have no induced
cycles of length more than four.

We prove the claim by induction on the depth of a decomposition tree
for $G$. If $G$ is the union of two tree-cographs $G_1$ and $G_2$
then the claim follows by induction since any induced cycle is
contained in one of $G_1$ and $G_2$. Assume $G$ is the join of two
tree-cographs $G_1$ and $G_2$. Assume that $G$ has an induced cycle
$C$ of length at least five. We may assume that $C$ has at least one
vertex in each of $G_1$ and $G_2$. As one of $G_1$ and $G_2$ has
more than two vertices of $C$, $C$ has a vertex of degree
at least three, which is a contradiction. \qed\end{proof}

\begin{lemma}
\label{complement of tree}
Let $T$ be a tree. Then $L(\Bar{T})^2$ is a clique.
\end{lemma}
\begin{proof} Consider two non-edges $\{a,b\}$ and $\{p,q\}$ of $T$. If the
non-edges share an endpoint then they are adjacent in $L(\Bar{T})^2$
since they are already adjacent in $L(\Bar{T})$. Otherwise, since
$T$ is a tree, at least one pair of $\{a,p\}$, $\{a,q\}$, $\{b,p\}$
and $\{b,q\}$ is a non-edge in $T$, otherwise $T$ has a 4-cycle. By
definition, $\{a,b\}$ and $\{p,q\}$ are adjacent in $L(\Bar{T})^2$.
\qed\end{proof}

\medskip

If $G$ is the union of two tree-cographs $G_1$ and $G_2$ then
\[\omega(L(G)^2)=\max\{\omega(L(G_1)^2),\omega(L(G_2)^2)\}.\]
The following lemma deals with the join of two tree-cographs.

\begin{lemma}
\label{join}
Let $P$ and $Q$ be tree-cographs and let $G$ be the join of
$P$ and $Q$. Let $X$ be the set of edges that have one endpoint
in $P$ and one endpoint in $Q$. Then
\begin{enumerate}[\rm (a)]
\item $X$ forms a clique in $L(G)^2$,
\item every edge of $X$ is adjacent
in $L(G)^2$ to every edge of $P$ and to every edge of $Q$, and
\item every edge of $P$ is adjacent in $L(G)^2$ to every edge of $Q$.
\end{enumerate}
\end{lemma}
\begin{proof} This is an immediate consequence of the definitions. \qed\end{proof}


For $k \geq 3$, a {\em $k$-sun} is a graph which consists of a
clique with $k$ vertices and an independent set with $k$ vertices.
There exist orderings $c_1,\ldots,c_k$ and $s_1,\ldots,s_k$ of the
vertices in the clique and independent set such that each $s_i$ is
adjacent to $c_i$ and to $c_{i+1}$ for $i=1,\ldots,k-1$ and such
that $s_k$ is adjacent to $c_k$ and $c_1$. A graph is {\em strongly
chordal} if it is chordal and has no $k$-sun, for $k \geq
3$~\cite{kn:farber}.

\begin{figure}[htb]
\setlength{\unitlength}{1.8pt}
\begin{center}
\begin{picture}(170,40)
\thicklines
\put(0,20){\circle*{2.0}}
\put(10,10){\circle*{2.0}}
\put(20,0){\circle*{2.0}}
\put(20,30){\circle*{2.0}}
\put(30,10){\circle*{2.0}}
\put(40,20){\circle*{2.0}}

\put(0,20){\line(1,-1){20}}
\put(0,20){\line(2,1){20}}
\put(10,10){\line(1,0){20}}
\put(20,0){\line(1,1){20}}
\put(40,20){\line(-2,1){20}}
\put(10,10){\line(1,2){10}}
\put(30,10){\line(-1,2){10}}

\put(67,20){\circle*{2.0}}
\put(77,10){\circle*{2.0}}
\put(87,30){\circle*{2.0}}
\put(97,10){\circle*{2.0}}
\put(107,20){\circle*{2.0}}

\put(67,20){\line(1,-1){10}}
\put(67,20){\line(2,1){20}}
\put(77,10){\line(1,0){20}}
\put(77,10){\line(1,2){10}}
\put(107,20){\line(-2,1){20}}
\put(97,10){\line(-1,2){10}}
\put(97,10){\line(1,1){10}}

\put(130,5){\circle*{2.0}}
\put(150,5){\circle*{2.0}}
\put(150,25){\circle*{2.0}}
\put(170,5){\circle*{2.0}}
\put(150,25){\line(-1,-1){20}}
\put(150,25){\line(0,-1){20}}
\put(150,25){\line(1,-1){20}}
\end{picture}
\caption{A $3$-sun, a gem and a claw}
\label{3-sun}
\end{center}
\end{figure}
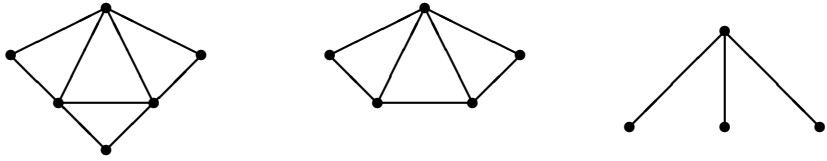

\begin{lemma}
\label{strongly chordal}
Let $T$ be a tree. Then $L(T)^2$ is strongly chordal.
\end{lemma}
\begin{proof} When $T$ is a tree then $L(T)$ is a blockgraph. Obviously,
blockgraphs are strongly chordal. Lubiw proves in~\cite{kn:lubiw}
that all powers of strongly chordal graphs are strongly chordal.
\qed\end{proof}

We strengthen the result of Lemma~\ref{strongly chordal} as follows.
{\em Ptolemaic graphs} are graphs that are both distance hereditary
and chordal~\cite{kn:howorka}. Ptolemaic graphs are gem-free chordal
graphs. The following theorem characterizes ptolemaic graphs.

\begin{theorem}[\cite{kn:howorka}]
A connected graph is ptolemaic if and only if
for all pairs of maximal cliques $C_1$ and $C_2$ with
$C_1 \cap C_2 \neq \es$, the intersection $C_1 \cap C_2$
separates $C_1 \setminus C_2$ from $C_2 \setminus C_1$.
\end{theorem}

\begin{lemma}
\label{ptolemaic}
Let $T$ be a tree. Then $L(T)^2$ is ptolemaic.
\end{lemma}
\begin{proof} Consider $L(T)$. Let $C$ be a block and let $P$ and $Q$ be two
blocks that each intersects $C$ in one vertex. Since $L(T)$ is
claw-free, the intersections of $P \cap C$ and $Q \cap C$ are
distinct vertices. The intersection of the maximal cliques $P \cup
C$ and $Q \cup C$, which is $C$, separates $P \setminus Q$ and $Q
\setminus P$ in $L(T)^2$. Since all intersecting pairs of maximal
cliques are of this form, this proves the lemma. \qed\end{proof}

\begin{corollary}
\label{char}
Let $G$ be a tree-cograph. Then $L(G)^2$ has a
decomposition tree with internal nodes labeled as join nodes
and union nodes and where the leaves are labeled as
ptolemaic graphs.
\end{corollary}

\medskip

{F}rom Corollary~\ref{char}
it follows that $L(G)^2$ is perfect~\cite{kn:chudnovsky},
that is, $L(G)^2$ has no odd holes or odd antiholes~\cite{kn:lovasz}.
This implies that the chromatic number of $L(G)^2$ is equal
to the clique number. Therefore, to compute the strong
chromatic index of a tree-cograph $G$ it suffices to compute the
clique number of $L(G)^2$.

\begin{theorem}
Let $G$ be a tree-cograph and let $T$ be a decomposition tree
for $G$. There exists a linear-time algorithm that computes
the strong chromatic index of $G$.
\end{theorem}
\begin{proof} First assume that $G=(V,E)$ is a tree. Then the strong chromatic
index of $G$ is
\begin{equation}
\label{form1}
s\chi^{\prime}(G)=\max \;\{\; d(x)+d(y)-1 \;|\; (x,y) \in E\;\}
\end{equation}
where $d(x)$ is the degree of the vertex $x$. To see this
notice that
Formula~(\ref{form1}) gives the clique number of $L(G)^2$.

\smallskip

Assume that $G$ is the complement of a tree.
By Lemma~\ref{complement of tree} the strong chromatic
index is the number of nonedges in $G$, which is
\[s\chi^{\prime}(G)=\binom{n}{2} - (n-1).\]

\smallskip

Assume that $G$ is the union of two tree-cographs $G_1$ and
$G_2$. Then, obviously,
\[s\chi^{\prime}(G)= \max \;\{\; s\chi^{\prime}(G_1),
\;s\chi^{\prime}(G_2)\;\}.\]

\smallskip

Finally, assume that $G$ is the join of two tree-cographs $G_1$
and $G_2$. Let $X$ be the set of edges of $G$ that have
one endpoint in $G_1$ and the other in $G_2$.
Then, by Lemma~\ref{join}, we have
\[s\chi^{\prime}(G) = |X|+s\chi^{\prime}(G_1) + s\chi^{\prime}(G_2).\]

The decomposition tree for $G$ has $O(n)$ nodes. For the trees the
strong chromatic index can be computed in linear time. In all other
cases, the evaluation of $s\chi^{\prime}(G)$ takes constant time. It
follows that this algorithm runs in $O(n)$ time, when a
decomposition tree is a part of the input. \qed\end{proof}

\section{Induced matching in tree-cographs}

Consider a strong edge coloring of a tree-cograph $G$. Then each
color class is an induced matching in $G$, which is an independent
set in $L(G)^2$~\cite{kn:cameron}. In this section we show that the
maximal value of an induced matching in $G$ can be computed in
linear time. Again, we assume that a decomposition tree is a part of
the input.

\begin{theorem}
\label{induced matching}
Let $G$ be a tree-cograph and let $T$ be a decomposition
tree for $G$. Then the maximal number of edges in an induced matching
in $G$ can be computed in linear time.
\end{theorem}
\begin{proof} In this proof we denote the cardinality of a maximum induced
matching in a graph $G$ by $i\nu(G)$.

First assume that $G$ is a tree. Since the maximum induced matching
problem can be formulated in monadic second-order logic, there
exists a linear-time algorithm to compute the cardinality of a
maximal induced matching in $G$ \cite{kn:brandstadt,kn:fricke}.

\smallskip

Assume that $G$ is the complement of a tree. By
Lemma~\ref{complement of tree} $L(G)^2$ is a clique. Thus the
cardinality of a maximum induced matching in $G$ is one if $G$ has a
nonedge and otherwise it is zero.

\smallskip

Assume that $G$ is the union of two tree-cographs $G_1$ and
$G_2$. Then
\[i\nu(G) = i\nu(G_1)+i\nu(G_2).\]

\smallskip

Assume that $G$ is the join of two tree-cographs
$G_1$ and $G_2$. Then
\[i\nu(G)= \max\;\{\;i\nu(G_1), \;i\nu(G_2), \;1\;\}.\]

This proves the theorem. \qed\end{proof}

\section{Permutation graphs}

A permutation diagram on $n$ points is obtained as follows.
Consider two horizontal lines $L_1$ and $L_2$ in the Euclidean plane.
For each line $L_i$ consider a linear ordering $\prec_i$ of
$\{1,\ldots,n\}$ and put points $1,\ldots,n$ on $L_i$ in this order.
For $k=1,\ldots,n$ connect the two points with the label $k$ by a
straight line segment.

\begin{definition}[\cite{kn:golumbic}]
A graph $G$ is a permutation graph if it is the intersection
graph of the line segments in a permutation diagram.
\end{definition}

\begin{figure}
\setlength{\unitlength}{1.8pt}
\begin{center}
\begin{picture}(150,45)
\thicklines
\put(10,20){\circle*{2.0}}
\put(20,10){\circle*{2.0}}
\put(20,30){\circle*{2.0}}
\put(40,10){\circle*{2.0}}
\put(40,30){\circle*{2.0}}
\put(10,20){\line(1,-1){10}}
\put(10,20){\line(1,1){10}}
\put(20,10){\line(1,0){20}}
\put(20,10){\line(0,1){20}}
\put(40,10){\line(0,1){20}}
\put(20,30){\line(1,0){20}}

\put(5,19){$4$}
\put(19,4){$2$}
\put(39,4){$3$}
\put(19,32){$5$}
\put(39,32){$1$}

\put(70,10){\circle*{1.8}}
\put(90,10){\circle*{1.8}}
\put(110,10){\circle*{1.8}}
\put(130,10){\circle*{1.8}}
\put(150,10){\circle*{1.8}}
\put(70,30){\circle*{1.8}}
\put(90,30){\circle*{1.8}}
\put(110,30){\circle*{1.8}}
\put(130,30){\circle*{1.8}}
\put(150,30){\circle*{1.8}}

\put(70,30){\line(2,-1){40}}
\put(90,30){\line(3,-1){60}}
\put(110,30){\line(-2,-1){40}}
\put(130,30){\line(0,-1){20}}
\put(150,30){\line(-3,-1){60}}
\put(69,4){$3$}
\put(89,4){$5$}
\put(109,4){$1$}
\put(129,4){$4$}
\put(149,4){$2$}
\put(69,32){$1$}
\put(89,32){$2$}
\put(109,32){$3$}
\put(129,32){$4$}
\put(149,32){$5$}
\end{picture}
\end{center}
\caption{A permutation graph and a permutation diagram}
\end{figure}
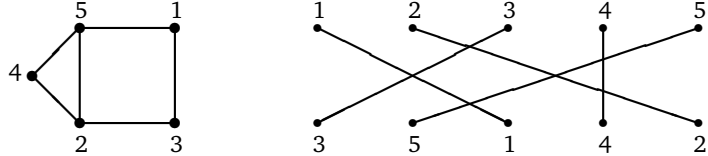

\medskip

Consider two horizontal lines $L_1$ and $L_2$ and on each line
$L_i$ choose $n$ intervals. Connect the left - and right endpoint
of the $k^{\mathrm{th}}$ interval on $L_1$ with the left - and
right endpoint of the $k^{\mathrm{th}}$ interval on $L_2$.
Thus we obtain a collection of $n$ trapezoids. We call this
a trapezoid diagram.

\begin{definition}
A graph is a trapezoid graph if it is the intersection
graph of a collection of trapezoids in a trapezoid diagram.
\end{definition}

\begin{lemma}
\label{permutation}
If $G$ is a permutation graph then $L(G)^2$ is a trapezoid graph.
\end{lemma}
\begin{proof} Consider a permutation diagram for $G$. Each edge of $G$
corresponds to two intersecting line segments in the diagram. The
four endpoints of a pair of intersecting line segments define a
trapezoid. Two vertices in $L(G)^2$ are adjacent exactly when the
corresponding trapezoids intersect (see Proposition~1
in~\cite{kn:cameron2}). \qed\end{proof}

\begin{theorem}
There exists an $O(n^4)$ algorithm that computes a
strong edge coloring in permutation graphs.
\end{theorem}
\begin{proof} Dagan, {\em et al.\/},~\cite{kn:dagan} show that a trapezoid
graph can be colored by a greedy coloring algorithm. It is easy to
see that this algorithm can be adapted so that it finds a strong
edge-coloring in permutation graphs. \qed\end{proof}

\begin{remark}
A somewhat faster coloring algorithm for trapezoid graphs
appears in~\cite{kn:felsner}. Their algorithm runs in
$O(n \log n)$ time where $n$ is the
number of vertices in the trapezoid graph.
An adaption of their algorithm yields a
strong edge coloring for permutation graphs that runs
in $O(m \log n)$ time, where $n$ and $m$ are the number of vertices
and edges in the permutation graph.
\end{remark}

\subsection{Bipartite permutation graphs}

A graph is a {\em bipartite permutation graph} if it is not only a
bipartite graph but also a permutation graph~\cite{kn:spinrad}. Let
$G=(A,B,E)$ be a bipartite permutation graph with color classes $A$
and $B$.

\begin{lemma}
\label{bip perm}
Let $G$ be a bipartite permutation graph. Then $L(G)^2$
is an interval graph.
\end{lemma}
\begin{proof} We first show that $L(G)^2$ is chordal. We may assume that
$L(G)^2$ is connected.

\smallskip

Let $x$ and $y$ be two non-adjacent vertices in a graph $H$.
An $x,y$-separator is a set $S$ of vertices which
separates $x$ and $y$ in distinct components. The separator
is a minimal $x,y$-separator if no proper subset of
$S$ separates $x$ and $y$.
A set $S$ is a
minimal separator if there exist non-adjacent vertices $x$ and
$y$ such that $S$ is a minimal $x,y$-separator.
Recall that Dirac characterizes chordal graphs by the property
that every minimal separator is a clique~\cite{kn:dirac}.

\smallskip

Consider the trapezoid diagram. Let $S$ be a minimal separator
in the trapezoid graph $L(G)^2$ and consider removing the trapezoids
that are in $S$ from the diagram. Every component of
$L(G)^2-S$ is a connected part in the diagram. Consider the
left-to-right ordering of the components in the diagram.
Since $S$ is a minimal separator
there must exist two {\em consecutive\/} components $C_1$ and $C_2$
such that every vertex of $S$ has a neighbor in both $C_1$
and $C_2$~\cite{kn:bodlaender}.

\smallskip

Assume that $S$ has two non-adjacent trapezoids $t_1$ and $t_2$.
Each of $t_i$ is characterized by two crossing line segments
$\{a_i,b_i\}$ of
the permutation diagram. Since $t_1$ and $t_2$ are not adjacent,
any pair of line-segments with one element in $\{a_1,b_1\}$ and the
other element in $\{a_2,b_2\}$ are parallel.

\smallskip

Each trapezoid $t_i$ intersects each component $C_1$ and $C_2$.
Since pairs of line-segments are parallel, we have that,
for some $i \in \{1,2\}$
\begin{enumerate}[\rm (1)]
\item $N_G(a_i) \cap C_1 \subseteq N_G(a_{3-i}) \cap C_1$,
\item $N_G(a_i) \cap C_1 \subseteq N_G(b_{3-i}) \cap C_1$,
\item $N_G(b_i) \cap C_1 \subseteq N_G(a_{3-i}) \cap C_1$ and
\item $N_G(b_i) \cap C_1 \subseteq N_G(b_{3-i}) \cap C_1$,
\end{enumerate}
and the reverse inequalities hold for $C_2$.
Each trapezoid $t_i$ has at least one line segment
of $\{a_i,b_i\}$ intersecting with a line segment of $C_i$.
By the neighborhood containments this implies that
$G$ has a triangle, which contradicts that $G$ is bipartite.
This proves that $S$ is a clique and by Dirac's characterization
$L(G)^2$ is chordal.

\smallskip

Lekkerkerker and Boland prove in~\cite{kn:lekkerkerker} that a graph
$H$ is an interval graph if and only if $H$ is chordal and $H$ has
no asteroidal triple. It is easy to see that a permutation graph has
no asteroidal triple~\cite{kn:kloks}. Cameron proves
in~\cite{kn:cameron2} (and independently Chang proves
in~\cite{kn:chang}) that $L(H)^2$ is AT-free whenever a graph $H$ is
AT-free. Thus, since $L(G)^2$ is chordal and AT-free, $L(G)^2$ is an
interval graph.

This proves the lemma. \qed\end{proof}

Chang proves in~\cite{kn:chang} that there exists a linear-time
algorithm that computes a maximum induced matching in bipartite
permutation graphs. We show that there is a simple linear-time
algorithm that computes the strong chromatic index of bipartite
permutation graphs.

\begin{theorem}
There exists a linear-time algorithm that computes the strong
chromatic index of bipartite permutation graphs.
\end{theorem}
\begin{proof} Let $G=(A,B,E)$ be a bipartite permutation graph and consider a
permutation diagram for $G$. Let $[a_1,\ldots,a_s]$ and
$[b_1,\ldots,b_t]$ be left-to-right orderings of the vertices of $A$
and $B$ on the topline of the diagram. Assume that $a_1$ is the
left-most endpoint of a line segment on the topline. We may assume
that the line segment of $a_1$ intersects the line segment of $b_1$.

\smallskip

Consider the set of edges in the maximal
complete bipartite subgraph $M$ in $G$
that consists of the following vertices.

\begin{enumerate}[\rm (a)]
\item $M$ contains $a_1$ and $b_1$,
\item $M$ contains all the vertices of $A$ of which the endpoint
on the topline is to the left of $b_1$,
\item $M$ contains all the vertices of $B$ of which the
endpoint on the bottom line is to the left of $a_1$.
\end{enumerate}

Notice that
$M$ is the set of edges in the complete bipartite subgraph in $G$
induced by
\[N[a_1] \cup N[b_1].\]
Extend the set of edges in $M$ with the edges in $G$ that have
one endpoint in $M$.
Call the set of edges in $M$ plus
the edges with one endpoint in $N(a_1) \cup N(b_1)$
the extension $\Bar{M}$ of $M$.
That is,
\[\Bar{M}=\{\;\{p,q\}\in E\;|\; p \in M \quad\text{or}\quad q \in M\;\}.\]
Notice that $\Bar{M}$ is the unique maximal clique that
contains the simplicial edge $\{a_1,b_1\}$ in $L(G)^2$.

\smallskip

The second maximal clique in $L(G)^2$ is found by
the process described above for the line segments induced
by $A \setminus \{a_1\} \cup B$.
Likewise, the third maximal clique in $L(G)^2$ is found by repeating
the process for the line segments induced by
$A \cup B \setminus \{b_1\}$.

\smallskip

Next, remove the vertices $a_1$ {\em and\/} $b_1$ and repeat the
three steps described above. It is easy to see that the list
obtained in this manner contains all the maximal cliques of
$L(G)^2$. Notice also that this algorithm can be implemented to run
in linear time. Since $L(G)^2$ is perfect the chromatic number is
equal to the clique number, so it suffices to keep track of the
cardinalities of the maximal cliques that are found in the process
described above. \qed\end{proof}

\section{Chordal bipartite graphs}

\begin{definition}[\cite{kn:golumbic2,kn:huang}]
A bipartite graph is chordal bipartite if it has no
induced cycles of length more than four.
\end{definition}

In contrast to bipartite permutation graphs,
$L(G)^2$ is not necessarily chordal when $G$ is chordal bipartite.
An example to the contrary is shown in Figure~\ref{counterexample}.
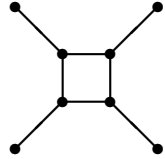
\begin{figure}
\setlength{\unitlength}{1.8pt}
\begin{center}
\begin{picture}(30,30)
\thicklines
\put(0,0){\circle*{2.0}}
\put(30,0){\circle*{2.0}}
\put(10,10){\circle*{2.0}}
\put(20,10){\circle*{2.0}}
\put(10,20){\circle*{2.0}}
\put(20,20){\circle*{2.0}}
\put(0,30){\circle*{2.0}}
\put(30,30){\circle*{2.0}}
\put(0,0){\line(1,1){10}}
\put(10,10){\line(1,0){10}}
\put(10,10){\line(0,1){10}}
\put(20,10){\line(1,-1){10}}
\put(20,10){\line(0,1){10}}
\put(10,20){\line(1,0){10}}
\put(10,20){\line(-1,1){10}}
\put(20,20){\line(1,1){10}}
\end{picture}
\caption{A chordal bipartite graph $G$ for
which $L(G)^2$ is not chordal.}
\label{counterexample}
\end{center}
\end{figure}

A graph is {\em weakly chordal} if it has no induced cycle of length
more than four or the complement of such a cycle~\cite{kn:hayward2}.
Weakly chordal graphs are perfect. Notice that chordal bipartite
graphs are weakly chordal.

Cameron, Sritharan and Tang prove in~\cite{kn:cameron3} (and
independently Chang proves in~\cite{kn:chang}) that $L(G)^2$ is
weakly chordal whenever $G$ is weakly chordal. Thus, if $G$ is
chordal bipartite then $L(G)^2$ is perfect and so, in order to
compute the strong chromatic index of $G$ it is sufficient to
compute the clique number in $L(G)^2$ (see also~\cite{kn:abueida}).

It is well-known that the clique number of
a perfect graph
can be computed in polynomial time~\cite{kn:grotschel}.\footnote{Actually,
this paper shows that for any graph $G$ with $\omega(G)=\chi(G)$
the values of these parameters can be determined in polynomial time.
The reason is that Lov\'asz' bound $\vartheta(G)$ for the Shannon capacity
of a graph can be computed in polynomial time
for all graphs, {\em via\/} the ellipsoid method, and
the parameter $\vartheta(G)$ is sandwiched between
$\omega(G)$ and $\chi(G)$.}
The algorithm presented in~\cite{kn:hayward} to compute the
clique number of weakly chordal graphs
runs in $O(n^3)$ time, where $n$ is the number
of vertices in the graph.

A direct application of their algorithm to solve the
strong chromatic index of chordal bipartite graphs $G$
involves computing the graph $L(G)^2$. This graph has $m$ vertices,
where $m$ is the number of edges in $G$. This gives a timebound
$O(n^6)$ for computing the strong chromatic index of a chordal
bipartite graph (see also~\cite{kn:abueida}).

In this section we show
that there is a more efficient method.

\begin{definition}[\cite{kn:yannakakis}]
A bipartite graph $G=(A,B,E)$ is a chain graph if there exists an
ordering $a_1, a_2, \ldots, a_{|A|}$ of the vertices in $A$ such
that $N(a_1) \subseteq N(a_2) \subseteq \cdots \subseteq
N(a_{|A|})$.
\end{definition}

Chain graphs are sometimes called {\em difference
graphs}~\cite{kn:hammer}. Equivalently, a graph $G=(V,E)$ is a chain
graph if there exists a positive real number $T$ and a real number
$w(x)$ for every vertex $x \in V$ such that $|w(x)| < T$ for every
$x \in V$ and such that, for any pair of vertices $x$ and $y$,
$\{x,y\} \in E$ if and only if $|w(x)-w(y)| \geqslant T$ $($see
Figure~\ref{fig difference} for an example$)$.

\begin{figure}[htb]
\begin{center}
\includegraphics[scale=1.0]{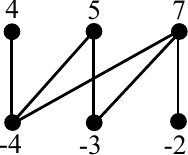}
\caption{\label{fig difference}A difference graph with $T=8$.}
\end{center}
\end{figure}

Chain graphs can be characterized in many ways~\cite{kn:hammer}. For
example, a graph is a chain graph if and only if it has no induced
$K_3$, $2K_2$ or $C_5$~\cite[Proposition~2.6]{kn:hammer}. Thus a
bipartite graph is a chain graph if it has no induced $2K_2$.
Abueida, {\em et al.\/}, prove that, if $G$ is a bipartite graph
that does not contain an induced $C_6$, then a maximal clique in
$L(G)^2$ is a maximal chain subgraph of $G$~\cite{kn:abueida}.
Notice that $C_6$ has three consecutive edges that form a clique in
$L(G)^2$, however, these edges do not form a chain subgraph.

\bigskip

Thus, computing the clique number of $L(G)^2$ for $C_6$-free
bipartite graphs $G$ is equivalent to finding the maximal number of
edges that form a chain graph in $G$. In~\cite{kn:abueida} the
authors prove that, if $G$ is a $C_6$-free bipartite graph, then
\[\chi(G^{\ast})=ch(G),\]
where $G^{\ast}$ is the complement of $L(G)^2$ and $ch(G)$ is the
minimum number of chain subgraphs of $G$ that cover the edges of
$G$.

\bigskip

An {\em antimatching} in a graph $G$ is a collection of edges which
forms a clique in $L(G)^2$~\cite{kn:mahdian}. It is easy to see that
finding a chain subgraph with the maximal number of edges in general
graphs is NP-complete. Mahdian mentions in his paper that the
complexity of maximum antimatching in simple bipartite graphs is
open~\cite{kn:mahdian}.

\begin{lemma}
Let $G$ be a bipartite graph. Finding a maximum set of edges that
form a chain subgraph of $G$ is NP-complete.
\end{lemma}
\begin{proof} Let $G=(A,B,E)$ be a bipartite graph. Let $C(G)$ be the graph
obtained from $G$ by making cliques of $A$ and $B$. Notice that $G$
is a chain graph if and only if $C(G)$ is chordal. Yannakakis shows
in~\cite{kn:yannakakis} that adding a minimum set of edges to $C(G)$
such that this graph becomes chordal is NP-complete.

\medskip

\noindent Consider the bipartite complement $G^{\prime}$ of $G$.
Adding a minimum set of edges such that $G$ becomes a chain graph is
equivalent to removing a minimum set of edges from $G^{\prime}$ such
that the remaining graph is a chain graph. This completes the proof.
\qed\end{proof}

\bigskip

In the following theorem we present our result for chordal bipartite
graphs.

\begin{theorem}
\label{chordal bip}
There exists an $O(n^4)$ algorithm that computes
the strong chromatic index of chordal bipartite graphs.
\end{theorem}
\begin{proof}
Let $G$ be chordal bipartite with color classes $C$ and $D$.
Consider the bipartite adjacency matrix $A$ in which rows
correspond with vertices of $C$ and columns correspond with vertices
of $D$. An entry of this matrix is one if the corresponding
vertices are adjacent and it is zero if they are not adjacent.

\medskip

\noindent
It is well-known that $G$ is chordal bipartite
if and only if $A$ is totally balanced.
Notice that a chain graph has a bipartite
adjacency matrix that is triangular. So we look for a
maximal submatrix of $A$ which is triangular after permuting
rows and columns.

\medskip

\noindent
Anstee and Farber and Lehel prove that a totally
balanced matrix, which has no repeated columns, can be completed
into a `maximal totally balanced matrix.'~\cite{kn:anstee2,kn:lehel},
If $A$ has $n$ rows
then this completing has $\binom{n+1}{2}+1$ columns.
The rows and columns of a maximal totally balanced matrix
can be permuted such that the adjacency matrix gets the following
form.

\begin{figure}[htb]
\begin{center}
\includegraphics[scale=0.5]{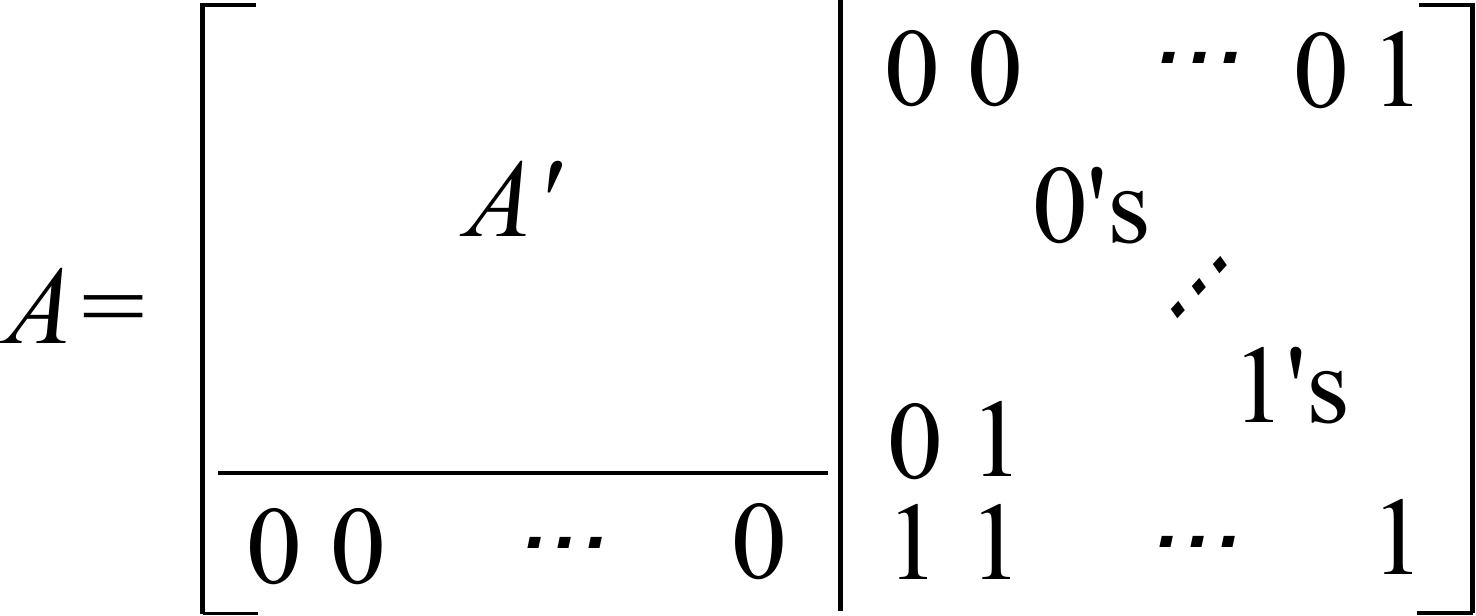}
\caption{\label{fig totallybalanced}A totally balanced matrix.}
\end{center}
\end{figure}

\medskip

\noindent
One can easily deal with repeated columns in $A$ by
giving the vertices a weight. When the matrix has the desired
form, one can easily find the maximal triangular submatrix
in linear time. Anstee and Farber give a rough bound of $O(n^5)$
to find the completion. But faster algorithms are given by
Paige and Tarjan and by Lubiw~\cite{kn:paige,kn:lubiw}.
\qed\end{proof}

\end{document}